\documentclass[AMA,STIX1COL]{WileyNJD-v2}
\usepackage{moreverb}

\newcommand\BibTeX{{\rmfamily B\kern-.05em \textsc{i\kern-.025em b}\kern-.08em
T\kern-.1667em\lower.7ex\hbox{E}\kern-.125emX}}

\articletype{Article Type}%

\received{<day> <Month>, <year>}
\revised{<day> <Month>, <year>}
\accepted{<day> <Month>, <year>}

\usepackage{geometry} 
\usepackage{amsmath}        
\usepackage{graphicx}
\usepackage{amsthm}
\usepackage{epstopdf}
\usepackage{amsfonts}
\usepackage{hyperref}
\makeatletter
\def\amsbb{\use@mathgroup \M@U \symAMSb}
\makeatother
\usepackage{bbold}
\usepackage{xcolor}
\usepackage{booktabs}
\usepackage{bm}
\usepackage{placeins}

\newtheorem{thm}{Theorem}
\newtheorem{cor}{Corollary}
\theoremstyle{definition}
\newtheorem{de}{Definition}
\newtheorem{re}{Remark}
\theoremstyle{lemma}

\newtheorem{lem}{Lemma}

\def\o{\omega}
\def\d{\delta}

\def\th{{\theta}}
\def\la{\lambda}
\def\E{\amsbb{E}}
\def\V{\amsbb{V}ar}
\def\C{\amsbb{C}ov}

\def\P{\amsbb{P}}
\def\R{\amsbb{R}}
\def\N{\amsbb{N}}
\def\nn{\nonumber}
\def\ga{\gamma}
\def\ro{{\cal R}_0}
\def\tS{\tilde S}
\def\D{\mathrm{d}}
\def\exp{\ensuremath{\rm{exp\,}}}
\def\log{\ensuremath{\rm{log\,}}}
\def\ln{\ensuremath{\rm{log\,}}}
\newcommand\ind[1]{\amsbb{I}_{#1}}

\def\rd{\color{black}}

\begin{document}

\title{Throwing Stones and Collecting Bones: Looking for Poisson-like Random Measures}

\author[1]{Caleb Deen Bastian*}

\author[2]{Grzegorz A Rempala}

\authormark{Bastian \textsc{et al}}

\address[1]{\orgdiv{Program in Applied and Computational Mathematics}, \orgname{Princeton University}, \orgaddress{\state{New Jersey}, \country{USA}}}

\address[2]{\orgdiv{Division of Biostatistics and Department of Mathematics}, \orgname{The Ohio State University}, \orgaddress{\state{Ohio}, \country{USA}}}

\corres{*Caleb Bastian, Program in Applied and Computational Mathematics, Washington Road, Fine Hall, Princeton University, Princeton, NJ. 08544 \email{cbastian@princeton.edu}}


\abstract[Abstract]{We show that in a broad class of random counting measures one may identify only three that are rescaled versions of themselves when restricted to a subspace. These are Poisson, binomial and negative binomial random measures. We provide some simple examples of possible applications of such measures.}



\maketitle

\section{Introduction} 

Random counting measures, also known as point processes, are the central objects of this note. For general introduction see, for instance monographs by\cite{kallenberg},\cite{cinlar}, or\cite{rm}. Random counting measures have numerous uses in statistics and applied probability, including representation and construction of stochastic processes, Monte Carlo schemes, etc. For example, the \emph{Poisson} random measure is a fundamental random counting measure that is related to the structure of L\'{e}vy processes, Markov jump processes, or the excursions of Brownian motion,  and is prototypical to the class of completely random (additive) random measures\citep{cinlar}.  In particular, it is also well know that  the Poisson random measure is self-similar in the sense of being invariant under  restriction to  a sub-space (invariant under thinning). The \emph{binomial} random measure is another fundamental random counting measure that underlies the theory of autoregressive integer-valued processes\citep{steutel1979discrete,binthin}. 

In this note we explore a broad class of random counting measures to identify those that share the Poisson self-similarity property and discuss their possible applications.  The paper is organized as follows. In the next section (Section 2) we provide necessary background and lay out the main mathematical results whereas in the following section (Section 3) we give examples of possible applications in different areas of modern sciences, from epidemiology to consumer research to traffic flows. 

The main result of the note is Theorem~\ref{thm:ptrms}, which  identifies in a broad class of random counting measures  those that are closed under restriction to subspaces, i.e. invariant under thinning. They are the Poisson, negative binomial, and binomial random measures. We show that the corresponding  counting distributions are the \emph{only} distributions in the power series family that are invariant under thinning. We also give simple examples to highlight calculus of PT random measures {\rd  and their possible applications}.  

\section{Throwing Stones and Looking for Bones}
\noindent Consider measurable space $(E,{\cal E})$ with some collection ${\bf X}=\{ X_i\}$ of iid random variables (stones) with law $\nu$ and some  non-negative integer valued random variable $K$ ($K\in \N_{\ge 0}=\N_{>0}\cup \{0\}$) {\rd with law $\kappa$} that is independent of $\bf X$ and has finite mean $c$. Whenever it exists, the variance of $K$ is denoted by $\d^2>0$. {\rd Let $\mathcal{E}_+$ be the set of positive $\mathcal{E}$-measurable functions.}

{\rd It is well known\cite{cinlar}  that the random counting measure $N$ on $(E,\mathcal{E})$ is uniquely determined  by  the pair of deterministic probability measures $(\kappa,\nu)$ through the} so-called {\em stone throwing construction} (STC) as follows. For every outcome  $\omega\in\Omega$
\begin{equation}\label{eq:stc} N_\o(A) =N(\o,A)=\sum_{i=1}^{K(\o)}\ind{A}(X_i(\o)) \quad \text{for }\,\,A \in {\cal E} \end{equation} {\rd  where $K$ has law $\kappa$,  the iid $X_1, X_2,\dotsb$ have law $\nu$ and $\ind{A}(\cdot)$ denotes  the indicator function for set $A$. Below we  write $N=(\kappa,\nu)$ to denote  the random measure $N$ determined by  $(\kappa,\nu)$ through STC. We note that $N$ may be also regarded as  a mixed binomial process\citep{rm}. In particular, when $\kappa$ is the Dirac measure, then $N$ is a binomial process\cite{rm}.}
Note that on any test function $f\in {\cal E}_+$ $$ N_\o f =\sum_i^{K(\o)} f\circ X_i(\o)= \sum_i^{K(\o)} f (X_i(\o)).$$ Below for brevity we write $Nf$, so that e.g., $N(A)=N\ind{A}$. It follows from the above and the independence of $K$ and 
 $\bf X$ that 
\begin{align}
\E Nf &=c\nu f \label{eq:01} \\
\V Nf &= c\nu f^2 + (\d^2-c) (\nu f)^2\label{eq:var2}
\end{align} 
and that the Laplace functional for $N$ is 
$$\E e^{-N f} =\E (\E e^{-f(X)})^K =\E (\nu e^{-f})^K=\psi(\nu e^{-f})$$ where $\psi(t)=\E\, t^K$ is the probability generating function (pgf) of $K$. {\rd  In what follows,  we will also  sometimes consider the alternate pgf (apgf)  defined as $\tilde{\psi}(t)=\E(1-t)^K.$}  Note also  that for any measurable partition of $E$, say $\{A,\ldots,B \} $, the joint distribution of the collection $N(A),\ldots,N(B)$ is for $i,\ldots, j \in \N$ and $i+\cdots+j =k$
\begin{align}\label{eq:02}
&\P(N(A)=i,\ldots, N(B)=j)\\
 &=\P(N(A)=i,\ldots, N(B)=j|K=k)\,\P(K=k)\nn\\
&=\frac{k!}{i!\cdots j!}\,\nu(A)^i\cdots \nu(B)^j\, \P(K=k).\nn
\end{align}

The following result extends construction  of a random measure $N=(K,\nu)$ to the case when the collection $\bf X$ is expanded to $({\bf X},{\bf Y})=\{(X_i,Y_i)\}$ where $Y_i$ is a random transformation of $X_i$.  Heuristically, $Y_i$ represents some properties (marks) of $X_i$. We assume that the conditional law of $Y$ follows some transition kernel according to $\P(Y\in B|X=x)=Q(x,B)$.
 \begin{thm}[Marked STC]\label{thm:marks} Consider random measure $N=(K,\nu)$ and the transition probability kernel  $Q$ from $(E, \cal E)$ into $(F, \cal F)$.  Assume that given the collection $\bf X$ the variables ${\bf Y}=\{Y_i\}$ are conditionally independent with $Y_i\sim Q(X_i,\cdot)$. Then $M=(K, \nu\times Q)$ is a random measure on $(E\times F, \cal E\otimes F)$. Here $\mu=\nu\times Q$ is understood as $\mu(dx,dy)=\nu(dx)Q(x,dy)$. Moreover, for any $f\in ({\cal E}\otimes {\cal F})_+$ 
 $$\E e^{-M f}=\psi(\nu e^{-g})$$ where $\psi(\cdot )$ is pgf of $K$ and $g\in {\cal E}_+$ satisfies  
$e^{-g(x)}= \int_F Q(x,dy)e^{-f(x,y)}.$ 
 \end{thm} 
The proof of this result is standard but for convenience we provide it in the appendix. For any $A\subset E$ with $\nu(A)>0$ define the conditional law $\nu_A$ by $\nu_A(B)= \nu (A\cap B)/\nu(A)$. The following is a simple consequence of Theorem~\ref{thm:marks} upon taking the transition kernel $Q(x,B)=\ind{A}(x)\,\nu_A(B)$.
\begin{cor}\label{cor:1} $N_A=(N\ind{A},\nu_A)$ is a well-defined random measure on the measurable subspace $(E\cap A, {\cal E}_A )$ where ${\cal E}_A=\{A\cap B: B\in {\cal E} \}$. Moreover, for any $f \in{\cal E}_+$ $$\E e^{-N_Af}=\psi(\nu e^{-f}\ind{A}+b)$$ where $b=1-\nu(A)$. \end{cor} 

In many practical situations one is interested in analyzing random measures of the form $N=(K, \nu\times Q)$ while having some information about the restricted measure $N_A=(N\ind{A}, \nu_A\times Q)$. Note that the  counting variable for $N_A$ is $K_A= N\ind{A}$, the original counting variable $K$ restricted to {\rd (thinned by)} the subset $A\subset E$.  The purpose of this note is to identify the families of counting distributions $K$ for which the family of random measures $\{N_A: A\subset E\}$ {\rd belongs to  the same  family of distributions}. We refer to such families of counting distributions as ``bones" and give their formal definition below. {\rd The term  reflects the   prototypical or foundational nature of these families within the class of random measures considered here.}  One obvious example is the Poisson family of distributions, but it turns out that there are also others. The definite result on the existence and uniqueness  of random measures based on such ``bones'' in a broad class is given in Theorem~\ref{thm:ptrms} of Section~\ref{sec:ex}.

\subsection{Subset Invariant Families (Bones) }
Let $N=(\kappa_\theta,\nu)$ be the random measure on $(E, {\cal E})$, where $\kappa_\theta$ is the distribution of $K$  parametrized by  $\th>0$, that is, $\P(K=k)_{k\ge 0}= (p_k(\th))_{k\ge 0}$ where we assume $p_0(\th)>0$. For brevity,  we write  below $K\sim \kappa_\theta$. 

Consider the family of random variables $\{ N\ind{A}: \, A\subset E \}$ and let $\psi_A(t)$ be the pgf of $N\ind{A}$ with $\psi_\th(t)=\psi_E(t)$ being the pgf of $K$ (since $N\ind{E}=K).$
Let $a=\nu(A)$, $b=1-a$ and note that $$ \psi_A(t) = \E (\E\, t^{\ind{A}})^K=\E (a t+b)^K=\psi_\th(at+b),$$ {\rd or equivalently,  in terms of apgf,   $\tilde{\psi}_A(t) = \tilde{\psi}_\theta(at)$.}


\begin{de}[Bones]\label{def:1} We say that the family {\rd $\{\kappa_\theta: \theta\in\Theta\}$} of counting probability measures is  {\em strongly invariant} with respect to the family $\{ N\ind{A}: \, A\subset E \}$ (is a {\em ``bone''})  if for any $0<a\le 1$ there exists a mapping $h_a: \Theta\to \Theta$ such that \begin{equation}\label{eq:2}\psi_\th(at+1-a)= \psi_{h_a(\theta)}(t).\end{equation}
\end{de}

{\rd  Note that in terms of apgf the above  condition becomes  simply $\tilde{\psi}_\theta(at)=\tilde{\psi}_{h_a(\theta)}(t)$.


}

In Table~\ref{tab:1} we give some examples of such invariant (``bone'') families. 

\begin{table}
\begin{center}
\begin{tabular}{l c c c }
\toprule
Name &Parameter $\th$ & $\psi_\th(t)$& $h_a(\theta)$ \\\midrule
Poisson & $\la$ & $\exp[{\th(t-1)}]$ &$a\th$ \\
Bernoulli & $p/(1-p)$ & $(1+\th t)/(1+\th)$& $a\th/(1+(1-a)\th)$\\
Geometric & $p$ & $(1-\th)/(1-t\th)$& $a\th/(1-(1-a)\th)$\\\bottomrule
\end{tabular}\caption{Some examples of ``bone'' distributions with corresponding pgfs and mappings of their canonical parameters.}\label{tab:1}
\end{center}
\end{table}

\subsection{Finding Bones in Power Series Family}\label{sec:ex}

Consider the family {\rd $\{\kappa_\theta: \theta\in\Theta\}$}  to be in the form of the non-negative power series (NNPS)  where \begin{equation}\label{eq:3} p_k(\th) = a_k \th^k /g(\th).\end{equation} and $p_0>0$. We call NNPS canonical if $a_0=1$. Setting $b=1-a$ we see that for {\rd canonical} NNPS the bone condition in Definition~\ref{def:1} becomes 
\begin{equation}\label{eq:4} g((at+b)\th)= g(b\th)g(h_a(\theta)t).\end{equation}
The following is a fundamental result on the existence of ``bones'' in the NNPS family.  
\begin{thm}[Bones in NNPS]\label{thm:bones} {\rd Let $\nu$ be diffuse (i.e., non-atomic)}. For canonical NNPS $\kappa_\theta$ satisfying additionally $a_1>0$, the relation \eqref{eq:4} holds iff $\log g(\th) =\th $ or $\log g(\th) =\pm {\rd c}\, \log(1\pm \th)$
 where ${\rd c}>0$. 
\end{thm}
\begin{proof} The proof follows from Lemma~1 in the appendix and the assumptions on NNPS family. 
\end{proof}
\begin{re}[Enumerating bones in NNPS] There are only three ``bones'' in canonical NNPS such that $a_1>0$, namely $\kappa_\theta$ is either Poisson, negative binomial or binomial.  Note that the entries in Table~\ref{tab:1} are all special cases. 
\end{re}
The ``bone'' families of distributions {\rd $\{\kappa_\theta: \theta\in\Theta\}$} are sometimes referred to as 
{\em Poisson-type} or $PT$\cite{jacobsen2018large}. We also refer to the random measures $N=(\kappa_\theta,\nu)$ where $\kappa_\theta$ is a ``bone'' family as 
{\em Poisson-type} or $PT$ random measures. The following is the main result of this note. 
\begin{thm}[Existence and Uniqueness of PT Random Measures]\label{thm:ptrms} Assume that $K\sim{\rd  \kappa_\theta}$ where pgf $\psi_\theta$ belongs to the canonical NNPS family of distributions and $\{0,1\}\subset supp(K)$. Consider the random measure $N=(\kappa_\theta,\nu)$ on the  space $(E, \cal E)$ and assume that {\rd $\nu$ is diffuse}. Then for any $A\subset E$ with $\nu(A)=a>0$ there exists a mapping $h_a:\Theta\rightarrow \Theta$ such that the restricted random measure is $N_A=(\kappa_{h_a(\theta)}, \nu_A)$, that is, \begin{equation}\label{eq:4a}\E e^{-N_Af}=\psi_{h_a(\theta)}(\nu_A e^{-f})\quad \text{for}\quad f\in {\cal E}_+\end{equation}  iff $K$ is Poisson, negative binomial or binomial. 
\end{thm}
\begin{proof} The sufficiency part follows by direct verification of \eqref{eq:4a} for $K$ Poisson, binomial, and negative binomial. The appropriate mappings are given in the last column of Table~1. The necessity part follows upon taking in \eqref{eq:4a} constant $f$ of the form $f(x)\equiv -\ln t$ for some $t\in(0,1]$ and applying Corollary~\ref{cor:1} and Theorem~\ref{thm:bones}. \end{proof}
\begin{re}\label{rem:1a} It follows from Theorem~\ref{thm:marks} that in Theorem~\ref{thm:ptrms} we may replace the laws $\nu$ and $\nu_A$ with $\nu\times Q$ and $\nu_A \times Q$, respectively.
\end{re}
Sometimes it may be more convenient to  parametrize PT distributions by their mean and variance (instead of $\th$) and write $PT(c,\d^2)$. The following is useful in computations related to PT random measures. 
\begin{re}[PT random measures can be thinned on average]\label{rem:2}
Note that if {\rd $N=(\kappa_\theta,\nu)$} is a PT random measure and $K\sim \kappa_\theta=PT(c,\d^2)$ then for any random variable $K_A=N\ind{A}$ where $A\subset E$ such that $\nu(A)=a>0$ it follows from  \eqref{eq:2} that \begin{align*} &\E K_A = a \E K =a c\\ &\E K_A(K_A-1)=a^2 \E K(K-1)=a^2(\d^2+c^2-c). \end{align*} 
\end{re}

{\rd 
\begin{re}[Atomic measure and a non-differentiable mapping]\label{rem:3} The iff result of Theorem~\ref{thm:ptrms} holds for diffuse measures $\nu$. When $\nu$ is atomic, the sufficiency part holds but the necessity part (uniqueness) fails if we also relax the differentiability condition for  the mapping $h$. 
To see this, consider the following simple example where we may construct a bone mapping for $K$ that is not PT. This example was generously pointed out to us by  one of  the reviewers.  Let $E=\{\blacklozenge, \blacksquare\}$ with $\nu\{\blacklozenge\}=1/2$. There are four subsets of $A\subseteq E$ with functions \[\psi_\varnothing(t)=1,\quad\psi_\blacklozenge(t)=\E((t+1)/2)^K=\psi_\blacksquare(t),\quad\psi_E(t)=\E t^K\]
For $A=\{\blacklozenge\}$ with $a=\nu\{\blacklozenge\}=1/2$, the restriction is $\psi_{\theta}(at+1-a)=\E((t+1)/2)^{K_\theta}$ with \[K_\theta=\begin{cases} \tilde{K}=\sum_{i=1}^K C_i & \theta=1\\K&\theta=2\\
\end{cases}\] where $\tilde{K}$ is the restricted or thinned version of $K$ by independent coin tosses $\{C_i\}$ (Bernoulli random variables) and $\Theta=\{1,2\}$. Then the bone condition \[\E((t+1)/2)^{K_\theta}=\E t^{K_{h_a(\theta)}}\] is satisfied with the mapping \[h_a(\theta) = \begin{cases} 1&a=1/2\\2&a=1.\end{cases}\]
\end{re}}

\section{Examples}

Below we discuss some simple examples of applications of PT random measures. The first one is an extension of the well-known construction for compound Poisson random measures. The second one is (to our knowledge) an original idea for application of binomial random measure to monitoring epidemics. Finally, the third one is an extension of a Poisson point process to a PT process  in  a particle system having birth and death dynamics, applied to traffic flows of spacecraft. 

\subsection{Compound PT Processes}\label{sec:cmpd} Assume that the number of customers and their arrivals times  over $n$ days form a PT random measure $(K,\nu)$ with $K\sim PT(c,\d^2)$ either Poisson or negative binomial.  Consider the associated mark random measure $N=(K,\nu\times Q\times Q_2)$ where $T\sim \nu$ gives  customer arrival times, and the transition kernels $Q(t,x)=P(X=x|T=t)$ and $Q_2(x,y)=\nu(y|X=x)$ describe, respectively,  customer's ``state'' $x=1,\ldots, s$ and his/her amount $Y$ spent at the store, so that each customer may be represented by the triple $(T,X,Y)$. We further assume that customers are independent with the conditional variable $(X|T=t)\sim Multinom(1,p^t_1,\ldots, p^t_s)$ and the conditional variable $(Y|X=x)$ with  mean $\alpha_x $ and variance $\beta^2_x$. Assume that we only have information about customers on a specific subset $A$  of $n$ days. We would like to decompose the average total amount $\E{\cal Z}$ spent by customers over the entire $n$ days period into two components corresponding to the observed and unobserved subsets ($A$ and $A^c$). Let therefore 
 \begin{equation}\label{eq:05} \E {\cal Z}= \E{\cal Z}_A+ \E {\cal Z}_{A^c}\end{equation} where ${\cal Z}_B$ is the total amount spend in time set $B\in \{A,A^c\}$. Recall PT random measure $N=(K,\tilde{\nu})$ where $\tilde{\nu}=\nu \times Q\times Q_2$, and consider two restricted measures  $N_B=(K_B,\tilde{\nu}_B)$ where $\tilde{\nu}_B=\nu_B\times Q\times Q_2$ for $B\in \{A,A^c\}$. 
Then 
$$
{\cal Z}= N f \quad \text{and}\quad {\cal Z}_B= N_B f, \quad B\in\{A,A^c \}
$$
where $f(t,x,y)=y$. {\rd By Theorem~\ref{thm:ptrms} ${\cal Z}_A$ and ${\cal Z}_{A^c}$ are also PT random measures with the corresponding $h_a(\th)$ transformation as presented in the last row of Table~\ref{tab:1}.}  Setting $b=\nu(B)$ and recalling Remark~\ref{rem:2}, it follows from  \eqref{eq:01} that for $B\in \{A, A^c \}$
 \begin{align*} \E {\cal Z}_B &= c b\,\tilde{\nu}_B f = c\tilde{\nu} f\ind{B} \\ 
&=c \int \ind{B}(t)\nu(dt) Q(t,dx) Q_2(x,dy) y\\ & =c \int_B \nu(dt) \sum_{x=1}^s p^t_x \alpha_x \end{align*} and 
\begin{align*} \V {\cal Z}_B &= cb\,\tilde{\nu}_B f^2 + b^2 (\tilde{\nu}_B f)^2 (\d^2-c) =c \tilde{\nu} f\ind{B} +(\d^2-c) (\tilde{\nu} f\ind{B})^2 \\
&= c\,\int_B \nu(dt) \sum_{x=1}^s p^t_x (\alpha_x^2+\beta_x^2)+ (\d^2-c) \left(\int_B \nu(dt) \sum_{x=1}^s p^t_x \alpha_x\right)^2.\end{align*} 
Similarly, we find $$\C({\cal Z}_A,{\cal Z}_{A^c})=(\d^2-c)\left(\int_A \nu(dt) \sum_{x=1}^s p^t_x \alpha_x\right)\left(\int_{A^c} \nu(dt) \sum_{x=1}^s p^t_x \alpha_x\right). $$

Consequently, from \eqref{eq:05} $$\E {\cal Z}= c\tilde{\nu} f\ind{A} +c\tilde{\nu} f\ind{A^c} = c\tilde{\nu} f= c \int \nu(dt) \sum_{x=1}^s p^t_x \alpha_x,$$ as well as \begin{align*}\V {\cal Z} 
&= \V {\cal Z}_A+\V {\cal Z}_{A^c} +2 \C({\cal Z}_A,{\cal Z}_{A^c})\\
&=c\,\int \nu(dt) \sum_{x=1}^s p^t_x (\alpha_x^2+\beta_x^2)+ (\d^2-c) \left(\int \nu(dt) \sum_{x=1}^s p^t_x \alpha_x\right)^2.
\end{align*} Note that last expression is equivalent to $c \tilde{\nu} f^2 +(\d^2-c) (\tilde{\nu} f)^2$ as obtained from \eqref{eq:01}. Note also that the term $\d^2-c $ is zero for $K$ Poisson (since then $N_A f$ and $N_{A^c} f$ are independent) but is strictly positive for $K$ negative binomial. Intuitively this implies that in this case the observed variable $N_A f$ carries some information about the unobserved $N_{A^c} f$. {\rd This idea appears to be closely related to negative binomial thinning\citep{steutel1979discrete}.  Observe that Theorem~\ref{thm:ptrms} states  that this type of  thinning operation cannot be extended to other  NNPS distributions.}

\subsection{SIR Epidemic Model} 
Assume that the independency of individuals $(U_i)$ surveyed for symptoms of infectious (or sexually transmitted) disease forms a random measure $N=(K,\nu\times Q )$ on the space $(E,\cal {E})$ where $E=\{(x,y): 0<x<y\}$. Each individual $U_i=(X_i,Y_i)$ is described by a pair  of infection and recovery times  and $K\sim Binom(n,p)$ where $n\ge 1$ and $p>0$ (to be specified later). Assume that at time $t>0$ the collection of labels $L_t(U_i)\in \{S,I,R\}$ for $i=1,\ldots,n$ is observed. 
 
To describe the relevant mean law $\nu\times Q$ consider a standard SIR model describing the evolution of proportions of susceptible ($S$) infectious ($I$) and removed ($R$) units according to the ODE system 
\begin{align}
 &\dot{S_t}=-\beta I_t S_t\label{eq:sir} \\
  & \dot{I}_t=\beta I_t S_t -\ga I_t\nonumber\\
  & \dot{R}_t=\ga I_t\nonumber
  \end{align}
with the initial conditions $S_0=1, I_0=\rho>0, R_0=0$. 
Define $\ro=\beta/\gamma>1$ and note that  
\begin{align}
S_t &=e^{-\ro R_t} \label{eq:sir2}\\
   I_t-\rho e^{-\ga t} &= -\int_0^t \dot{S_u} e^{-\ga(t-u)} du. \label{eq:sir2a} 
 \end{align}
Interpreting \eqref{eq:sir} as the {\em mass transfer model} (see, \cite{khudabukhsh2019survival}) with initial mass $S_0=1$, the function  $S_t$  is  the probability of an initially susceptible unit remaining uninfected at time $t>0$.  Since $S_t+I_t+R_t=1+\rho$ and $I_\infty=0$ then $S_\infty=1-\tau$ where $\tau\in (0,1)$ is the solution of 
\begin{equation}\label{eq:tau}1-\tau=e^{-\ro (\tau+\rho)}.\end{equation} By the law of total probability 
$S_t=\tau \tilde{S}_t+1-\tau$
where $\tS_t$ is a proper survival function conditioned on the fact that the unit will eventually get infected, an event with probability $\tau<1$ given by \eqref{eq:tau}. Note that the Lebesgue density function of the proper conditional distribution function $1-\tS_t$ is simply 
\begin{equation}\label{eq:f} 
\nu(x)=-\dot{S}_x/\tau.
\end{equation} 
Define now $\tau\tilde{I}_t:=I_t-\rho e^{-\ga t}$ and note that from \eqref{eq:sir2a} and the last equation in \eqref{eq:sir} we may interpret $\ga\tilde{I}_t$ as the Lebesgue density of the (conditional) recovery time $t$ given by the Lebesgue density of the sum of two independent random variables, one of them being exponential with rate $\gamma$. Hence, we may define the mean law $\nu\times Q$ by taking \eqref{eq:f} along with the  transition kernel $Q(x,\cdot)$ in the form of the shifted exponential Lebesgue density
 $$Q(x,y)=H_x(y)\sim Exp(\gamma)\ind{\{x<y\}}(y).$$
 To complete the definition of $N$  take $K\sim Binom(n,p)$ with $p=\tau$ defined in \eqref{eq:tau} so that $\E K=n\tau$. 

For fixed $t>0$  let the sets $E^t_S=\{(x,y): x>t \}$, $E^t_I=\{(x,y): x\le t<y \}$ and $E^t_R=\{(x,y): x<y\le t\}$ define the $t$-induced partition of the space $E$. 
Define the label on the $i$-th individual observed at time $t$ as 
$$ L_t(U_i) =\begin{cases}
  S& \text{if $U_i \in E_S$}, \\
  I& \text{if $U_i \in E_I$}, \\
   R& \text{if $U_i \in E_R$}.
\end{cases}$$
 Setting $k=k_S+k_I+k_R$ from \eqref{eq:02} we obtain that 
\begin{align}\label{eq:part}
&\P(N(E^t_S)=k_S,N(E^t_I)=k_I, N(E^t_R)=k_R)\\
 &=\P(N(E^t_S)=k_S,N(E^t_I)=k_I, N(E^t_R)=k_R|K=k)\,\P(K=k)\nn\\
&=\frac{n!}{k_S!\,k_I!\,k_R!\,(n-k)!}\,(\tau \tilde{S}_t)^{k_S} (\tau \tilde{I}_t)^{k_I} \tau^{k_R}(1- \tilde{S}_t- \tilde{I}_t)^{k_R} \, (1-\tau)^{n-k}.\nn
\end{align} 
Since the overall count of susceptible labels is $k_S+n-k$, marginalizing over the unobserved counts $k_S$ and $k$ gives the final distribution of $I,R$ labels among $n$ individuals at time $t$ 
$$\P(N(E^t_I)=k_I, N(E^t_R)=k_R)=\frac{n!}{k_I! k_R! (n-k_I-k_R)! } (\tau \tilde{I}_t)^{k_I} (1- S_t- \tau\tilde{I}_t)^{k_R} S_t^{n-k_I-k_R}. $$
Hence, it follows in particular that for the $i$-th individual its label probabilities at $t$ are $\P(L_t(U_i)=S)=S_t$, $ \P(L_t(U_i)=I)=\tau \tilde{I}_t$ and $\P(L_t(U_i)=R)=1-S_t-\tau\tilde{I}_t$. 

Let $A=(0,t]$ and define the conditional infection Lebesgue density by rescaling \eqref{eq:f} $$\nu_A(x)=\nu(x)\ind{\{x<t\}}(x)/\nu(A)=-\dot{S}_x\ind{\{x<t\}}(x)/(1-S_t).$$ {\rd Then by Theorem~\ref{thm:ptrms} and Remark~\ref{rem:1a} the restricted random measure $N_A=(K_A,\nu_A\times Q)$ is a binomial random measure and }
 according to Remark~\ref{rem:2} $$\E K_A=n(1-\tilde{S}_t)\tau=n(1-S_t)$$
$$\V K_A =n\tau(1-\tau)(1-\tilde{S}_t)^2+n\tau\tilde{S}_t(1-\tilde{S}_t)=nS_t(1-S_t), $$ so we see that $K_A\sim Binom(n,1-S_t).$

\subsection{Spacecraft Traffic Flows} Consider a particle system of vehicles moving about in $E\subset\R^3$. We are interested in the locations of the vehicles in space and time. We assume the vehicles form an independency, i.e. are mutually independent, implied by weak gravitational interaction, and their configuration forms a {\rd random counting measure $N$ with number of vehicles $K\sim\kappa_\theta$}. Particle system ideas have been applied in air traffic control, for example in an ``interacting'' particle system of aircraft for estimating collision probabilities\citep{atc}. {\rd  We consider the scenario of space traffic control, now in its infancy, by taking $E$ as the Solar System and vehicles as spacecraft (such as satellites, rockets, space planes, space stations, probes, etc), although these ideas may be readily applied to air traffic control, which is in a mature state. 

A key issue for space traffic control is modeling the counts of the particle system in various subspaces $\{N\ind{A}: A\subset E\}$, such as in regions of interest\cite{space2}, e.g., space traffic control thinning (restriction) of the particle system into orbital regimes has been considered a topical issue in a recent Presidential Memorandum\cite{space}. Traffic flows can be subject to complex dynamics, with varying degrees of ``interactions'' among spacecraft (in the sense of correlated counts in time and space). 

An obvious extension of the Poisson point process model used in \citep{traffic}  is to use random counting measures closed under thinning with general covariance, i.e. PT random measures. We discuss the role of PT random measures in describing the dynamics of the arrivals of spacecraft into subspaces of time and space.

}

To describe the atomic structure of the particle system, first we label the spacecraft with integers $i$ in $\N_{>0}$. Let $X_i$ be the initial location of spacecraft $i$ in $(E,\mathcal{E})$ and $Y_i=(Y_i(t))_{t\in\R_+}$ be its motion in $(F,\mathcal{F})$. Each $Y_i$ is a stochastic process with state-space $(E,\mathcal{E})$, a path in space and time called a \emph{world line} (also known as a \emph{trajectory} or \emph{orbit}) and regarded as a random element of the function space $(F,\mathcal{F})=(E,\mathcal{E})^{\R_+}$. The quantity $Y_i(t)$ is the location of spacecraft $i$ in $(E,\mathcal{E})$ at time $t$, where $Y_i(0)=X_i$ is the initial location. Therefore each spacecraft $i$ is described by a pair $(X_i,Y_i)$. Assume $X$ has law $\nu$ and the conditional variable $(Y | X=x)$ has transition probability kernel $Q(x,B)=\P(Y\in B|X=x)$ for $B\in\mathcal{F}$. 
We construct random measures from independencies using STC. Let $K\sim\kappa_\theta$ where $\kappa_\theta$ is PT. The independency $\mathbf{Y}=\{Y_i\}$ forms a PT random measure $M=(\kappa_\theta,\mu)$ on $(F,\mathcal{F})$ through STC as \[M(A) = \sum_{i=1}^{K}\ind{A}(Y_i)\quad\text{for}\quad A\in\mathcal{F}\] with mean measure $\mu=\nu Q$ defined by \[\mu(A) = \int_{E}\nu(\D x)Q(x,A)\quad\text{for}\quad A\in\mathcal{F}.\] Consider the mapping $h:F\mapsto E$ as $h(w) = w(t)$ for $w\in F$ and $t$ fixed. The PT image random measure $N_t=M\circ h^{-1}=(\kappa_\theta,\mu_t)$ on $(E,\mathcal{E})$ is formed by $\mathbf{Y}(t)=\{Y_i(t)\}$ through STC as \begin{equation}\label{eq:lat}N_t(A) = \sum_{i=1}^K\ind{A}(Y_i(t))\quad\text{for}\quad A\in\mathcal{E}\end{equation} with mean $\mu_t=\mu\circ h^{-1}=\nu P_t$ defined by \begin{equation}\label{eq:mut}\mu_t(A) = \int_{E}\nu(\D x)\int_F Q(x,\D w)\ind{A}(w(t)) = \int_{E}\nu(\D x)P_t(x,A)\quad\text{for}\quad A\in\mathcal{E}\end{equation} and the $\{Y_i(t)\}$ having conditional distributions $\{P_t(X_i,\cdot)\}$ defined by {\rd transition kernel} \[P_t(x,A)=\int_F Q(x,\D w)\ind{A}(w(t))=\P(w\in F: w(t)\in A,\, w(0)=x)\quad\text{for}\quad A\in\mathcal{E}.\] The \eqref{eq:lat} and \eqref{eq:mut} $N_t = (\kappa_\theta, \mu_t)$ defines an immortal particle system on $(E,\mathcal{E})$. {\rd The family of transition kernels $(P_t)_{t\in\R_+}$ is the transition semigroup in the theory of Markov processes\cite{cinlar}. Queries about the particle system $f\in\mathcal{E}_+$ form random variables $N_tf$ with mean \eqref{eq:01} and variance \eqref{eq:var2}. 

The concept of thinning is well established for particle systems, such as in the Bienaym\'{e}-Galton-Watson branching process literature\cite{feller} as well as in the analysis of count time-series using PT thinning operators\cite{thin2}. Space traffic control thinning (restriction) of the particle system into disjoint subspaces is a key operation. Using PT random measures, let $A\subset E$ with $\mu_t(A)=a>0$ be a subspace and $N_A=(\kappa_{h_a(\theta)},\mu_A)$ be the restricted random measure of $N_t$ with $\mu_A(\cdot)=\mu_t(A\cap\cdot)/\mu_t(A)$. Theorem~\ref{thm:ptrms} says all such thinnings $\{N_A: A\subset E\}$ are PT. Hence Theorem~\ref{thm:ptrms} is archetypical for space traffic control. Moreover, the PT family members identified in Theorem~\ref{thm:ptrms} convey distinct dynamic meanings for the counting process of the particle system, reflected in their covariances. For Poisson, the counts of spacecraft arrivals in disjoint subspaces are independent and Markov and correspond to low-density flows of freely passing spacecraft\cite{traffic}. For binomial, the counts in disjoint subspaces are negatively correlated and are identified to following behaviors, platoons, or congestion\cite{traffic2}. For negative binomial, counts in disjoint subspaces are positively correlated and are identified to flows having cycles, control intersections, or contagion\cite{traffic2}. These ideas carry over to the random variables $\{N_t f: f\in\mathcal{E}_+\}$.
}

Additional frills for the particle system include a notion of birth and death, manufacture and destruction respectively. Death is achieved through a single point extension of the state-space to contain a point $\partial$ outside of $E$ called a cemetery with measure space $(\bar{E},\bar{\mathcal{E}})$, where $\bar{E}=E\cup\{\partial\}$ and $\bar{\mathcal{E}}=\mathcal{E}\cup\{A\cup\{\partial\}: A\in\mathcal{E}\}$. The world line space becomes $(F,\mathcal{F})=(\bar{E},\bar{\mathcal{E}})^{\R_+}$. Manufacturing is the notion of an arrival time for each spacecraft $T_i$ on $(\R,\mathcal{B}_{\R})$, independent of spacecraft location or motion. $Y_i(t)$ is the location of spacecraft $i$ in $(\bar{E},\bar{\mathcal{E}})$ at time $T_i + t$, and $Y_i(0)=X_i$ is the (manufacturing) location at time $T_i$. For Earth or Moon manufacturing, the motion $(Y(t))_{t>0}$ involves moving the manufactured spacecraft to a spaceport, launching, and bringing into orbit. Some spacecraft undergo repeated orbital maneuvers, such as landing at a spaceport, launching, and bringing into orbit, repeating many times\citep{spacex}. Note that under this setup, the measure $P_t$ is defective on $(E,\mathcal{E})$ as some spacecraft that are manufactured are destroyed with probability $1-P_t(E)$. 

The independency $(\mathbf{T},\mathbf{X},\mathbf{Y})=\{(T_i,X_i,Y_i)\}$ forms the random measure $N=(K,\eta\times\nu\times Q)$ on $\R\times E\times F$ through STC as \[N(A) = \sum_{i=1}^K\ind{A}(T_i,X_i,Y_i)\quad\text{for}\quad A\in\mathcal{B}_{\R}\otimes\mathcal{E}\otimes\mathcal{F}.\] To describe spacecraft manufactured and not yet destroyed, let \[h(s,x,w)= \begin{cases} 
    (s,x,w(t-s))& \text{for}\quad s\le t\\
    (s,x,\partial) & \text{for}\quad s>t
   \end{cases}\] and put $N\circ h^{-1}$ as the image of $N$ under $h$. Then spacecraft manufactured and not yet destroyed at time $t$ are represented by the trace of $N\circ h^{-1}$ on $E$. This is formed by $(\mathbf{T},\mathbf{X},\mathbf{Y}(t-\mathbf{T}))=\{(T_i,X_i,Y_i(t-T_i))\}$ through STC as \begin{equation}\label{eq:nt}N_t(A) = \sum_{i=1}^K \ind{(-\infty,t]\times E\times A}(T_i,X_i,Y_i(t-T_i))\quad\text{for}\quad A\in\mathcal{E}\end{equation} with mean $\mu_t$ defined by \begin{equation}\label{eq:mut2}\mu_t(A) =\int_{(-\infty,t]}\eta(\D s)\int_{E}\nu(\D x)\int_F Q(x,\D w)\ind{A}(w(t-s))\quad\text{for}\quad A\in\mathcal{E}.\end{equation} {\rd Other elaborations of the model $N_t=(\kappa_\theta,\mu_t)$ include expanding the state-space of the particle system to provision additional mark spaces, such as radiation detection and crew and passenger health monitoring systems for each spacecraft.}

\section{Discussion and Conclusions}

It is well- known that the PT distributions are invariant under thinning\cite{thin2,thin}. The ``if'' part of our Theorem~\ref{thm:ptrms}  gives a different  proof  of this result in terms of  a certain functional equation called the ``bone'' condition.  To the best of our knowledge the ``only if'' part of the theorem is novel. Therefore,  the main result is the definite one on the existence and uniqueness of PT random measures as random counting measures invariant under thinning.

We characterize PT distributions as those discrete distributions whose generating functions  satisfy the ``bone'' condition. Hence we can refer to the PT distributions as the ``bone'' class of distributions. It turns out that there are other characterizations for PT distributions aside from the ``bone'' condition. For example, the PT distributions arise when considering discrete distributions whose mass functions obey a certain recursive relation and are called the \emph{Panjer} or \emph{(a,b,0)} classes of distributions\citep{panjer}. Yet another (similar) recursive relation involving mass functions recapitulates the PT distributions as the \emph{Katz} family of distributions\citep{katz}. Another route to attaining the PT distributions is starting with and generalizing the Poisson distribution to the Conway-Maxwell-Poisson distribution, each PT member being a special or limiting case\citep{cmp}. These highlight how PT distributions possess rich structure and are independently retrievable using multiple distinct hypotheses. 


Given the ubiquity of random count data, PT random measures have wide utility in the sciences. We  illustrate this  with several examples.  First, we give an extension to the compound model applied to modeling the amount of money spent by customers in a store, using compound Poisson and negative binomial random measures. We also give an application to monitoring epidemics, showing that the popular SIR model has the structure of a binomial random measure. Finally, we give an application to closed particle systems, highlighting how the distinct covariances of the PT random measures confer multiple dynamical meanings to the particle system. %

\section{Acknowledgements} This work was funded by the US National Science Foundation grant DMS 1853587. The authors are indebted to the two reviewers for many  valuable comments and suggestions that helped  improve the original manuscript. This work does not have any conflicts of interest.

\section{Bibliography}

\section{Appendix Section}

\appendix

\section{Proofs}
{\bf Proof of Theorem~\ref{thm:marks}}. 
It suffices to verify the claimed identity for the Laplace functional of $M=(K, \nu\times Q)$ with arbitrary $f\in ({\cal E}\otimes {\cal F})_+$ as it  will in particular imply the existence of $M$. To this end consider 
$$\E e^{-M f} =\E (\E e^{-f(X,Y)})^K =\psi((\nu\times Q) e^{-f})$$ where $\psi(\cdot )$ is pgf of $K$.
Since $$(\nu\times Q)\, e^{-f} =\E \int _F Q(X, dy)\,e^{-f}\circ (X,y)= \nu e^{-g},$$  where $g\in {\cal E}_+$ is defined by 
$$e^{-g(x)}= \int_F Q(x,dy)e^{-f(x,y)}, $$ therefore 
$$\E e^{-M f}=\psi((\nu\times Q) e^{-f})=\psi(\nu e^{-g})=\E e^{-Ng}<\infty.$$\qed

\noindent {\bf Proof of Theorem~\ref{thm:bones}}. The result follows from the following lemma.
\begin{lem}[Modified Cauchy Equation]\label{lem:1} 
Assume that $f(t)$  is twice continuously differentiable in some neighborhood of the origin, satisfies $f(0)=0$ and $f^\prime(0)>0$ as well as 
\begin{equation}\label{eq:5} f(s+t)-f(s)=f(h(s)\,t)\end{equation} where $h(s)$ is $t$ free.
Then $f$ is of the form $ f(t) =At$ or $f(t)=B\log (1+At)$ for some $A,B\ne 0$. Moreover $h(s)=f^\prime(s)/f^\prime(0)$.
\end{lem}
\begin{proof}
Differentiating \eqref{eq:5} with respect to $t$ we obtain
\begin{equation}\label{eq:5a} f'(s+t) =h(s)f'(h(s)\,t). \end{equation} Taking the above at $t=0$ and denoting $C_1=f'(0)>0$ gives 
\begin{equation}\label{eq:6}
h(s)=f'(s)/C_1.
\end{equation}
Differentiating \eqref{eq:5} with respect to $s$  yields likewise  (note that $h$ is differentiable in view of \eqref{eq:6})
$$ f'(s+t)=f'(s)+t\,h'(s)f'(h(s)\,t). $$ Equating the two right hand side expressions and using \eqref{eq:6} we have 
\begin{align*} C_1h(s)+t\,h'(s)f'(h(s)\,t) &=h(s)f'(h(s)\,t)\\
f'(h(s)t)= \frac{C_1}{1-t\frac{h'(s)}{h(s)}}.
\end{align*} 
In the last expression we take now $s=0$, denote  $C_2=h'(0)$ and consider two cases according to  $C_2=0$ and $C_2\ne 0$. Since by \eqref{eq:6} $h(0)=1$, for the case $C_2=0$  \begin{equation}\label{eq:7}
f(t)=At\end{equation}  where ($A=C_1$) and we have one solution. 
 Consider now $C_2\ne 0$,  then 
 $$f'(t)= \frac{C_1}{1-C_2t}$$
 and hence the general form of $f$ when it is not linear is 
 $$f(t) = B\log (1+At)$$ where $B=-C_1/C_2$ and $A=-C_2$. This as well as \eqref{eq:7} and \eqref{eq:6} give the hypothesis of the theorem. 
 \end{proof} 
 

\end{document}